\documentclass[reqno, oneside, 12pt]{amsart}

\usepackage[letterpaper]{geometry}
\usepackage{amsmath}
\usepackage{amssymb}
\usepackage{enumerate}
\usepackage{verbatim}
\usepackage{mathrsfs}
\usepackage{mathtools}

\mathtoolsset{centercolon}

\usepackage{setspace}

\numberwithin{equation}{section}

\newtheorem{thm}{Theorem}[section]
\newtheorem{cor}[thm]{Corollary}
\newtheorem{lem}[thm]{Lemma}
\newtheorem{prop}[thm]{Proposition} 

\newtheorem{conjec}[thm]{Conjecture}

\theoremstyle{remark}
\newtheorem*{rem*}{Remark}

\date{}

\author{Heidi Goodson}
\address{Department of Mathematics and Statistics, Haverford College; 370 Lancaster Avenue, Haverford, PA 19041 USA}
\email{hgoodson@haverford.edu}

\title{A Complete Hypergeometric Point Count Formula for Dwork Hypersurfaces}

\begin{document}

\begin{abstract}
We extend our previous work on hypergeometric point count formulas by proving that we can express the number of points on families of Dwork hypersurfaces
$$X_{\lambda}^d: \hspace{.1in} x_1^d+x_2^d+\ldots+x_d^d=d\lambda x_1x_2\cdots x_d$$
over finite fields of order $q\equiv 1\pmod d$ in terms of Greene's finite field hypergeometric functions. We prove that when $d$ is odd, the number of points can be expressed as a sum of hypergeometric functions plus $(q^{d-1}-1)/(q-1)$ and conjecture that this is also true when $d$ is even. The proof rests on a result that equates certain Gauss sum expressions with finite field hypergeometric functions. Furthermore, we discuss the types of hypergeometric terms that appear in the point count formula and give an explicit formula for Dwork threefolds.
\end{abstract}

\maketitle

\section{Introduction}

The motivation for this work comes from a particular family of elliptic curves. For $\lambda\not= 0,1$ we define an elliptic curve in the Legendre family by 
$$E_\lambda : y^2 = x(x-1)(x-\lambda).$$

We compute a period integral associated to the Legendre elliptic curve given by integrating the nowhere vanishing holomorphic $1$-form $\omega=\frac{dx}{y}$ over a $1$-dimensional cycle containing $\lambda$. This period is a solution to a hypergeometric differential equation and can be expressed as the classical hypergeometric series
$$\pi=\int_{0}^{\lambda} \frac{dx}{y}={}_2F_1\left(\left.\begin{array}{cc}
                \frac12&\frac12	\\
		{}&1
               \end{array}\right|\lambda\right).$$
See the exposition in \cite{Clemens} for more details on this. \\

We now specialize to the case where $\lambda\in\mathbb Q\setminus\{0,1\}$. Koike \cite[Section 4]{Koike1995} showed that, for all odd primes $p$, the trace of Frobenius for curves in this family can be expressed in terms of Greene's hypergeometric function \cite{Greene}
\begin{equation}\label{eqn:EC2F1Greene}
a_{E_\lambda}(p)=-\phi(-1)p\cdot {}_2F_1\left(\left.\begin{array}{cc}
                \phi&\phi\\
		{}&\epsilon
               \end{array}\right|\lambda\right)_{p},\end{equation}
where $\epsilon$ is the trivial character and $\phi$ is a quadratic character modulo $p$.\\ 

Note the similarity between the period and trace of Frobenius expressions: the period is given by a classical hypergeometric series whose arguments are the fractions with denominator 2 and the trace of Frobenius is given by a finite field hypergeometric function whose arguments are characters of order 2. This similarity is to be expected for curves. Manin proved in \cite{Manin} that the rows of the Hasse-Witt matrix of an algebraic curve are solutions to the differential equations of the periods. In the case where the genus is 1, the Hasse-Witt matrix has a single entry: the trace of Frobenius. Igusa showed in \cite{Igusa1958} that, for odd primes $p$, the trace of Frobenius is congruent to a classical hypergeometric expression 
\begin{equation}\label{eqn:EC2F1Classical}a_{E_\lambda}(p)\equiv(-1)^{\frac{p-1}{2}}{}_2F_1\left(\left.\begin{array}{cc}
                \frac12&\frac12	\\
{}&1
               \end{array}\right|\lambda\right) \pmod{p}.
\end{equation}                           
Furthermore, in Corollary 3.2 of \cite{GoodsonDwork2017}, we show that the finite field and classical ${}_2F_1$ hypergeometric expressions in Equations  \ref{eqn:EC2F1Greene} and \ref{eqn:EC2F1Classical} are congruent modulo $p$ for odd primes. This result would imply merely a congruence between the finite field hypergeometric function expression and the trace of Frobenius. The fact that Koike showed that we actually have an equality is very intriguing and leads us to wonder for what other varieties this type of equality holds.\\

Further examples of a correspondence between arithmetic properties of varieties and finite field hypergeometric functions have been observed for algebraic curves \cite{Swisher2016, Fuselier10, Lennon1, Mortenson2003a, Vega2011} and for particular Calabi-Yau threefolds \cite{AhlgrenOno00a, McCarthy2012b}. For example, Fuselier \cite{Fuselier10} gave a finite field hypergeometric trace of Frobenius formula for elliptic curves with $j$-invariant $\frac{1728}{t}$, where $t\in \mathbb F_p \setminus \{0,1\}$.  Lennon \cite{Lennon1} extended this by giving a hypergeometric trace of Frobenius formula that does not depend on the Weierstrass model chosen for the elliptic curve. In \cite{AhlgrenOno00a}, Ahlgren and Ono gave a formula for the number of $\mathbb F_p$ points on a modular Calabi-Yau threefold. Greene's hypergeometric functions were also used for modular form results in \cite{AhlgrenOno00a,FOP1, FOP2,Fuselier10,Kilbourn2006,Lennon2,Mortenson2003a,Mortenson2005}. We extended this work in \cite[Theorem 1.1]{GoodsonDwork2017} by showing that the number of points on the family of Dwork K3 surfaces over finite fields can be expressed in terms of Greene's finite field hypergeometric functions. In \cite[Theorems 1.3, 1.4]{GoodsonDwork2017} we also gave a formula for the number of points on Dwork K3 surfaces in terms of p-adic hypergeometric functions, which were defined by McCarthy in \cite{McCarthy2013}. We note that these two $p$-adic results were recently extended to higher dimensional Dwork hypersurfaces in \cite{McCarthy2016arxiv}.  \\

In \cite{GoodsonDwork2017}, we began to develop hypergeometric point count formulas for higher dimensional Dwork hypersurfaces in the following theorem.
\begin{thm}\cite[Theorem 8.1]{GoodsonDwork2017}\label{thm:DworkHyp}
 Let $q\equiv 1\pmod d$, $t=\frac{q-1}{d}$, and $T$ be a generator for $\widehat{\mathbb F_q^{\times}}$. The number of points over $\mathbb F_q$ on the Dwork hypersurface is given by
\begin{align*}
\#X_{\lambda}^d(\mathbb F_q)&=\frac{q^{d-1}-1}{q-1} + q^{d-2}\cdot{}_{d-1}F_{d-2}\left(\left.\begin{array}{cccc}
                T^t&T^{2t}&\ldots &T^{(d-1)t}\\
		{} &\epsilon&\ldots&\epsilon
               \end{array}\right|\frac{1}{\lambda^d}\right)_q\\
  &\hspace{1in} -\frac1q\sum_{j=0}^{d-1}g\left(T^{\frac{j(q-1)}{d}}\right)^d-\frac{1}{q-1}\sum_{\overline j}\prod_{j_i}g\left(T^{\frac{j_i(q-1)}{d}}\right),
\end{align*}
where the last sum is over all $d$-tuples $\overline j =(j_1,\ldots,j_d)$ with $0\leq j_i\leq d-1$ and $\sum{j_i}\equiv 0 \pmod d$.    
\end{thm}

In this paper, we prove that when $d$ is odd, the remaining Gauss sum terms can be written in terms of hypergeometric functions.
\begin{thm}\label{thm:DworkHypOdd}
 Let $d$ be an odd integer, and let $q\equiv 1\pmod d$. The number of points over $\mathbb F_q$ on the Dwork hypersurface can be expressed as $(q^{d-1}-1)/(q-1)$ plus a sum of Greene's finite field hypergeometric functions.
\end{thm}
We prove this result in Section \ref{subsec:d_odd}. We also conjecture that this is true for $d$ even.
\begin{conjec}\label{conjec:DworkHypEven}
 Let $d$ be an even integer, and let $q\equiv 1\pmod d$. The number of points over $\mathbb F_q$ on the Dwork hypersurface can be expressed as $(q^{d-1}-1)/(q-1)$ plus a sum of Greene's finite field hypergeometric functions.
\end{conjec}
We discuss both progress on this conjecture and obstructions to a complete result in Section \ref{subsec:d_even}. Note that this conjecture has been verified for $d=4$, i.e. for Dwork K3 surfaces (see \cite{GoodsonDwork2017}), and should follow more generally from  McCarthy's work in \cite{McCarthy2016arxiv} and Miyatani's work in \cite{Miyatani2015}. Furthermore, in Section \ref{sec:TypesofTerms} we discuss when certain types of hypergeometric terms will appear in the point count formulas. We show that this often depends on the parity of $d$.\\

Explicit formulas for the number of points on Dwork hypersurfaces are useful for many reasons, such as determining local zeta factors in the global Hasse-Weil Zeta function from which arithmetic L-functions arise. Salerno \cite{Salerno2013} developed point count formulas for Dwork hypersurfaces in terms of Katz's \cite{Katz1990} hypergeometric functions.   We are particularly interested in point count formulas written in terms of Greene's finite field hypergeometric formulas for the following reason. We observe an interesting phenomenon with certain periods associated to Dwork hypersurfaces. These periods can be written in terms of classical hypergeometric series, a fact that was first noted by Dwork in \cite{Dwork1969}. Interestingly, the hypergeometric expressions for the periods and the point counts ``match" in the sense that fractions with denominator $a$ in the classical series coincide with characters of order $a$ in the finite field hypergeometric functions. We saw this matching of expressions with Dwork K3 surfaces in \cite{GoodsonDwork2017}, and here we show that this occurs for Dwork threefolds as well.
\begin{thm}\label{thm:ThreefoldPointCount}
Let $q=p^e$ be a prime power such that $q\equiv 1\pmod 5$, $t=\frac{q-1}{5}$, and $T$ be a generator for $\widehat{\mathbb F_q^{\times}}$. Then
\begin{align*}
\#X_{\lambda}^5(\mathbb F_q)&=\frac{q^{4}-1}{q-1}+ 24 q^2\delta(1-\lambda^5)+ q^{3}{}_{4}F_{3}\left(\left.\begin{array}{cccc}
                T^t&T^{2t}&\ldots &T^{4t}\\
		{} &\epsilon&\ldots&\epsilon \end{array}\right|\frac{1}{\lambda^5}\right)_q\\\\
              \hspace{1in}&+20q^2{}_{2}F_{1}\left(\left.\begin{array}{cc}
                T^{2t}&T^{3t}\\
		{} &\epsilon               \end{array}\right|\frac{1}{\lambda^5}\right)_q+20q^2{}_{2}F_{1}\left(\left.\begin{array}{cc}
                T^{t}&T^{4t}\\
		{} &\epsilon               \end{array}\right|\frac{1}{\lambda^5}\right)_q\\\\
		\hspace{1in}&+30q^2{}_{2}F_{1}\left(\left.\begin{array}{cc}
                T^{t}&T^{3t}\\
		{} &T^{4t}               \end{array}\right|\frac{1}{\lambda^5}\right)_q+30q^2{}_{2}F_{1}\left(\left.\begin{array}{cc}
                T^{t}&T^{2t}\\
		{} &T^{3t}               \end{array}\right|\frac{1}{\lambda^5}\right)_q,
\end{align*}
where $\delta(x)=1$ if $x=0$ and $\delta(x)=0$ otherwise.
\end{thm}
We prove this result in Section \ref{sec:DworkThreefold}. Meanwhile, Candelas, De La Ossa, and Rodriguez-Villegas show in \cite{COV1} that the periods (that are holomorphic at $\lambda=0$) of the Dwork threefold are given by the classical series
\begin{align*}
    &{}_{4}F_{3}\left(\left.\begin{array}{cccc}
                1/5&2/5&3/5&4/5\\
		{} &1&1&1 \end{array}\right|\frac{1}{\lambda^5}\right),\\\\
          &{}_{2}F_{1}\left(\left.\begin{array}{cc}
                2/5&3/5\\
		{} &1              \end{array}\right|\frac{1}{\lambda^5}\right),\hspace{.2in} {}_{2}F_{1}\left(\left.\begin{array}{cc}
                1/5&4/5\\
		{} &1            \end{array}\right|\frac{1}{\lambda^5}\right),\\\\
		&{}_{2}F_{1}\left(\left.\begin{array}{cc}
                1/5&3/5\\
		{} &4/5          \end{array}\right|\frac{1}{\lambda^5}\right),\hspace{.2in} {}_{2}F_{1}\left(\left.\begin{array}{cc}
                1/5&2/5\\
		{} &3/5          \end{array}\right|\frac{1}{\lambda^5}\right),
\end{align*}
with multiplicities matching the coefficients in the point count formula. There are an additional 24 periods that appear only when $\lambda^5=1$, which corresponds to the term $24 q^2\delta(1-\lambda^5)$ in the point count formula. Our work in Section 3 of \cite{GoodsonDwork2017} shows that the ${}_4F_3$ and the first two ${}_2F_1$ classical hypergeometric series and the corresponding finite field hypergeometric functions are congruent modulo $p$ (when $q=p$), however we do no yet have a congruence or identity for the remaining terms.\\

It should be noted that Dwork hypersurface families are particularly nice to work with because of their large group of automorphisms. In general, one should expect many more terms in the point count formula and more periods for a Calabi-Yau manifold. The expected number comes from the Hodge structure, which gives us information about the complex structure of the moduli space and the Betti numbers and dictates the number of expected periods.\\ 

For example, as discussed in Section 3 of \cite{COV1}, the non-trivial Hodge numbers of the Dwork threefold are $h^{1,1}=1$ and $h^{2,1}=101$. This gives a Betti number of $B_3=2(1+h^{2,1})=204$. Thus, we should expect there to be 204 periods of the holomorphic $(3,0)$-form. However, the automorphism group reduces this number to 101. The amount of computation that needs to be done is further whittled down since many of the periods are equivalent modulo the Jacobian ideal. In fact, when grouped together in this way, the number of sets corresponding to a particular period is the same as the coefficient of the ``matching" term in the point count formula.\\

This phenomenon holds true for Dwork K3 surfaces, too. Here, we have Betti number $B_2=22$, so we should expect there to be 22 periods of the holomorphic $(2,0)$-form. However, there are in fact 16 periods and they fall into three distinct types (see Dwork's exposition in Chapter 6 of \cite{Dwork1969b} for more on this). We note that the three types are expressible as classical hypergeometric series and that these series match the hypergeometric functions in the point count formula. We expect that there should be a similar matching of periods and terms in the point count for higher dimensional Dwork hypersurfaces.

\section{Preliminaries}
\subsection{Hypergeometric Series and Functions}
\label{sec:HGF}
We start by recalling the definition of the classical hypergeometric series
\begin{equation}\label{eqn:classicalHGF}
 {}_{n+1}F_{n}\left(\left.\begin{array}{cccc}
                a_0&a_1&\ldots& a_n\\
		{}&b_1&\ldots,& b_n
               \end{array}\right|x\right) = \displaystyle\sum_{k=0}^{\infty}\dfrac{(a_0)_k\ldots(a_n)_k}{(b_1)_k\ldots(b_n)_kk!}x^k,
\end{equation}
where $(a)_0=1$ and $(a)_k=a(a+1)(a+2)\ldots(a+k-1)$.  \\

In his 1987 paper \cite{Greene}, Greene introduced a finite field, character sum analogue of classical hypergeometric series that satisfies similar summation and transformation properties. Let $\mathbb F_q$ be the finite field with $q$ elements, where $q$ is a power of an odd prime $p$. If $\chi$ is a multiplicative character of $\widehat{\mathbb F_q^{\times}}$, extend it to all of $\mathbb F_q$ by setting $\chi(0)=0$. For any two characters $A,B$ of $\widehat{\mathbb F_q^{\times}}$ we define the normalized Jacobi sum by

\begin{equation}\label{eqn:normalizedjacobi}
 \binom{A}{B}:=\frac{B(-1)}{q}\sum_{x\in\mathbb F_q} A(x)\overline B(1-x) = \frac{B(-1)}{q}J(A,\overline{B}),
\end{equation}
where $J(A,B)=\sum_{x\in \mathbb F_q} A(x)B(1-x)$ is the usual Jacobi sum.\\

For any positive integer $n$ and characters $A_0,\ldots, A_n,B_1,\ldots, B_n$ in $\widehat{\mathbb F_q^{\times}}$, Greene defined the finite field hypergeometric function ${}_{n+1}F_n$ over $\mathbb F_q$ by
\begin{equation}\label{eqn:HGFdef}
 {}_{n+1}F_{n}\left(\left.\begin{array}{cccc}
                A_0,&A_1,&\ldots,&A_n\\
		{} &B_1,&\ldots,&B_n
               \end{array}\right|x\right)_q = \displaystyle\frac{q}{q-1}\sum_{\chi}\binom{A_0\chi}{\chi}\binom{A_1\chi}{B_1\chi}\ldots\binom{A_n\chi}{B_n\chi}\chi(x).
\end{equation}

In the case where $n=1$, an alternate definition, which is in fact Greene's original definition, is given by
\begin{equation}\label{eqn:2F1def}
 {}_{2}F_{1}\left(\left.\begin{array}{cc}
                A&B\\
		{} &C
               \end{array}\right|x\right)_q = \epsilon(x)\frac{BC(-1)}{q}\sum_yB(y)\overline{B}C(1-y)\overline{A}(1-xy).
\end{equation}

Note that Greene's finite field hypergeometric functions were defined independently of those defined by Katz \cite{Katz1990} and McCarthy \cite{McCarthy2012c}, relations between them have been demonstrated in \cite{McCarthy2012c}.

\subsection{Gauss and Jacobi Sums}
\label{sec:GaussJacobi}
Unless otherwise stated, information in this section can be found in Ireland and Rosen's text \cite[Chapter 8]{Ireland}.\\

Let $q=p^e$. We define the standard trace map $\text{tr}:\mathbb F_q\rightarrow \mathbb F_p$ by 
$$\text{tr}(x)=x+x^p+\ldots + x^{p^{e-1}}.$$
Let $\pi\in\mathbb C_p$ be a fixed root of $x^{p-1}+p=0$ and let $\zeta_p$ be the unique $p^\text{th}$ root of unity in $\mathbb C_p$ such that $\zeta_p \equiv 1+\pi \pmod{\pi^2}$. Then for $\chi \in\widehat{\mathbb F_q^{\times}}$ we define the Gauss sum $g(\chi)$ to be
\begin{equation}\label{eqn:GaussSum}
 g(\chi):=\sum_{x\in\mathbb F_q} \chi(x)\theta(x),
\end{equation}
where we define the additive character $\theta$ by $\theta(x)=\zeta_p^{\text{tr}(x)}$. Note that if  $\chi$ is nontrivial then $g(\chi)g(\overline\chi)=\chi(-1)q$.\\ 

We have the following connection between Gauss sums and Jacobi sums. For non-trivial characters $\chi$ and $\psi$ on $\mathbb F_q$ whose product is also non-trivial,
$$J(\chi,\psi)=\frac{g(\chi)g(\psi)}{g(\chi\psi)}.$$
More generally, for non-trivial characters $\chi_1,\ldots,\chi_n$ on $\mathbb F_q$ whose product is also non-trivial,
$$J(\chi_1,\ldots,\chi_n)=\frac{g(\chi_1)\cdots g(\chi_n)}{g(\chi_1\cdots \chi_n)}.$$

Another important product formula is the Hasse-Davenport formula.

\begin{thm}\label{thm:HasseDavenport}\cite[Theorem 10.1]{Lang1990}
 Let $m$ be a positive integer and let $q$ be a prime power such that $q\equiv 1 \pmod m$. For characters $\chi,\psi\in\widehat{\mathbb F_q^{\times}}$ we have 
$$\prod_{i=0}^{m-1}g(\chi^i\psi)=-g(\psi^m)\psi^{-m}(m) \prod_{i=0}^{m-1}g(\chi^i).$$
\end{thm}

We will use the following specialization of this.
\begin{cor}\label{cor:HasseDavenportGeneral}
 $$g(T^{dj})=\frac{\prod_{i=0}^{d-1} g(T^{it+j})}{T^{-dj}(d)\prod_{i=1}^{d-1}g(T^{it})},$$
 where $q\equiv 1\pmod{d}$, $t=\frac{q-1}{d}$, and $T$ is a generator for $\widehat{\mathbb F_q^{\times}}$.
\end{cor}
\begin{proof}
 This follows from Theorem \ref{thm:HasseDavenport} using $m=d, \chi=T^t,$ and $\psi=T^j$.
\end{proof}

\section{Gauss Sum Identities}

The following is a Gauss sum relation that generalizes Proposition 2.5 in \cite{GoodsonDwork2017}.

\begin{prop}\label{prop:gaussproduct}
Let $q$ be a prime power such that $q\equiv 1\pmod 4$, $t=\frac{q-1}{4}$, and $T$ be a generator for $\widehat{\mathbb F_q^{\times}}$. Let $a,b$ be multiples of $t$. Then, for $\lambda^4\not=1$,
$$\displaystyle \sum_{j=0}^{q-2} g(T^{j+a})g(T^{-j+b})T^j(-1)T^{4j}(\lambda)=(q-1)g(T^{a+b})T^b(-1)T^{-(a+b)}(1-\lambda^4).$$
\end{prop}

\begin{proof}
We start by expanding the Gauss sums
\begin{align*}
   \sum_{j=0}^{q-2} g(T^{j+a})g(T^{-j+b})T^j(-1)T^{4j}(\lambda) &= \sum_{j=0}^{q-2}T^j(-\lambda^4) \left(\sum_{x\in\mathbb F_q}T^{j+a}(x)\theta(x)\right) \left(\sum_{y\in\mathbb F_q}T^{-j+b}(y)\theta(y)\right)\\
   &= \sum_{j=0}^{q-2}T^j(-\lambda^4) \sum_{x,y\in\mathbb F_q}T^{j+a}(x)T^{-j+b}(y)\theta(x+y)\\ 
   &= \sum_{j=0}^{q-2}T^j(-\lambda^4) \sum_{x,y\in\mathbb F_q^{\times}}T^{j}(x/y)T^{a}(x)T^{b}(y)\theta(x+y)\\
   &= \sum_{x,y\in\mathbb F_q^{\times}}T^{a}(x)T^{b}(y)\theta(x+y)\sum_{j=0}^{q-2}T^j\left(-\tfrac{\lambda^4x}{y}\right).
  \end{align*}
Note that $\sum_{j=0}^{q-2}T^j\left(-\tfrac{\lambda^4x}{y}\right)=0$ unless $-\frac{\lambda^4x}{y}=1$, in which case the sum equals $q-1$. So, we let $x=-\frac{y}{\lambda^4}$ to get
\begin{align*}
 \sum_{j=0}^{q-2} g(T^{j+a})g(T^{-j+b})T^j(-1)T^{4j}(\lambda)   &= (q-1)\sum_{y\in\mathbb F_q^{\times}}T^{a}\left(-\tfrac{y}{\lambda^4}\right)T^{b}(y)\theta\left(-\tfrac{y}{\lambda^4}+y\right)\\
 &= (q-1)\sum_{y\in\mathbb F_q^{\times}}T^{a}\left(-\tfrac{y}{\lambda^4}\right)T^{b}(y)\theta\left(y(-\tfrac{1}{\lambda^4}+1)\right).
  \end{align*}
  
Then we perform the change of variables $y\to y(-1/\lambda^4+1)^{-1}$ to get
\begin{align*}
\sum_{y\in\mathbb F_q^{\times}}T^{a}\left(\tfrac{-y}{\lambda^4-1}\right)T^{b}\left(\tfrac{y}{-1/\lambda^4+1}\right)\theta(y) &=T^{-a}(1-\lambda^4)T^{-b}\left(\tfrac{-1}{\lambda^4}+1\right)\sum_{y\in\mathbb F_q^{\times}}T^{a+b}(y)\theta(y)\\
 &=T^{-a}(1-\lambda^4)T^{-b}\left(\tfrac{\lambda^4-1}{\lambda^4}\right)\sum_{y\in\mathbb F_q^{\times}}T^{a+b}(y)\theta(y)\\
 &=T^{b}(-\lambda^4)T^{-(a+b)}(1-\lambda^4)g(T^{a+b}).
\end{align*}
Hence, 
\begin{align*}
 \displaystyle \sum_{j=0}^{q-2} g(T^{j+a})g(T^{-j+b})T^j(-1)T^{4j}(\lambda)&=  (q-1)T^{b}(-\lambda^4)T^{-(a+b)}(1-\lambda^4)g(T^{a+b})\\
 &=(q-1)g(T^{a+b})T^b(-1)T^{-(a+b)}(1-\lambda^4),
\end{align*}
where the last equation holds because $b$ is a multiple of $t$ and $T^{4t}(\lambda)=1$.

\end{proof}

This proposition generalizes nicely for Gauss sum expressions of a particular form. We first note that by combining Theorem 3.13 and Definition 3.5 of \cite{Greene}, we can express finite field hypergeometric functions in the following way. For characters $A_0, \ldots, A_n, B_1,\ldots, B_n$ over $\mathbb F_q$ and $x\in \mathbb F_q^\times$,

\begin{align}\label{eqn:FnDefinition}
    &{}_{n+1}F_{n}\left(\left.\begin{array}{cccc}
                A_0& A_1 &\ldots &A_n\\
		{} &B_1&\ldots&B_n               \end{array}\right|x_0\right)_q = \frac{\prod_{j=1}^n A_jB_j(-1)}{q^n} \nonumber\\
		&\hspace{.5 in} \cdot \sum_{x_i}A_1(x_1)\overline{A_1}B_1(1-x_1)\cdots A_n(x_n)\overline{A_n}B_n(1-x_n)\overline{A_0}(1-x_0x_1\cdots x_n)
\end{align}

The following theorem relates Gauss sum expressions and Greene's hypergeometric functions. Gauss sum expressions of this form appear in the Dwork hypersurface point count formula of Section \ref{sec:GeneralHypFormula}.

\begin{thm}\label{thm:GaussSumIDFn}
Let $q$ be a prime power such that $q\equiv 1 \pmod d$ and $t=\frac{q-1}{d}$. For a positive integer $n$ and for $1\leq i \leq n$, let $a_i, b_i$ be integer multiples of $t$, not all 0, such that $a_k\not=-b_j$ for all $k,j$ and $b_k\not=b_j$ for $k\not=j$. Then, for $\lambda \not=0$,  
\begin{align*}
&\frac{1}{q-1}\sum_{j=0}^{q-2}\left(\prod_{i=1}^n g(T^{a_i+j})\prod_{i=1}^n T^{-b_i+j}(-1)g(T^{b_i-j})\right) T^{j}(\lambda^d)\\
&\hspace{.5in}=T^{m}(-1)G\cdot q^{n-1}{}_{n}F_{n-1}\left(\left.\begin{array}{cccc}
                T^{b_n+a_1}& T^{b_1+a_1} &\ldots &T^{b_{n-1}+a_1}\\
		{} &T^{a_1-a_2}&\ldots&T^{a_1-a_n}              \end{array}\right|\lambda^d\right)_q,
\end{align*}
where $m=\sum_1^n a_i -\sum_1^{n-1} (b_i +a_1)$ and $G=g(T^{a_2+b_1})\cdots g(T^{a_n+b_{n-1}})g(T^{b_n+a_1})$.
\end{thm}

\begin{proof}

We start by assuming $a_1\leq\ldots \leq a_n$ and $b_1\leq\ldots \leq b_n$ since the above Gauss sum expression is independent of this ordering. Recalling that $g(\chi)=\sum_x \chi(x)\theta(x)$ we can write

\begin{align*}
&\frac{1}{q-1}\sum_{j=0}^{q-2}\left(\prod_{i=1}^n g(T^{a_i+j})\prod_{i=1}^n T^{-b_i+j}(-1)g(T^{b_i-j})\right) T^{j}(\lambda^d)\\
 &\hspace{1in}=\frac{T^{m'}(-1)}{q-1}\sum_{x_i, y_i\in\mathbb F_q}T^{a_1}(x_1)\cdots T^{a_n}(x_n)T^{b_1}(y_1)\cdots T^{b_n}(y_n)\\
&\hspace{2in}\cdot\theta\left(\sum x_i +\sum y_i\right)\sum_{j=0}^{q-2}T^j\left(\tfrac{(-1)^nx_1\cdots x_n\lambda^d}{y_1\cdots y_n}\right),
\end{align*}
where $x_i,y_i\not=0$ and $m'=-(b_1+\ldots +b_n)$. Note that $\sum T^j\left(\tfrac{(-1)^nx_1\cdots x_n\lambda^d}{y_1\cdots y_n}\right)=q-1$ if $(-1)^nx_1\cdots x_n\lambda^d/y_1\cdots y_n=1$ and equals 0 otherwise. Letting $x_1=\tfrac{(-1)^ny_1\cdots y_n}{x_2\cdots x_n\lambda^d}$ and recalling that $T^{dt}(\lambda)=1$ yields the following expression
\begin{align*}
&T^{m'+a_1}(-1)\sum_{x_i, y_i}T^{a_2-a_1}(x_2)\cdots T^{a_n-a_1}(x_n)T^{b_1+a_1}(y_1)\cdots T^{b_n+a_1}(y_n)\\
&\hspace{.5in}\cdot\theta\left(\tfrac{(-1)^ny_1\cdots y_n}{x_2\cdots x_n\lambda^d} +x_2+\ldots+ x_n+y_1+\ldots+ y_n \right),
\end{align*}
where we sum over all $x_i, y_i$ except $x_1$.\\

Our goal now is to get the above expression in terms of Gauss sums and multiplicative characters. We perform the following changes of variables.
$$y_1\rightarrow y_1x_2,\hspace{.1in} y_2\rightarrow y_2x_3,\hspace{.1in} \ldots \hspace{.1in}, \hspace{.1in}y_{n-1}\rightarrow (-1)^ny_{n-1}x_n\lambda^d.$$
This yields the expression
\begin{align*}
&T^{m'+a_1}(-1)\sum_{x_i, y_i}T^{a_2+b_1}(x_2)\cdots T^{a_n+b_{n-1}}(x_n)T^{b_1+a_1}(y_1)\cdots T^{b_n+a_1}(y_n)\\
&\hspace{.5in}\cdot\theta(y_1\cdots y_n +x_2+\ldots+ x_n+y_1x_2+y_2x_3+\ldots+(-1)^ny_{n-1}x_n\lambda^d+ y_n).
\end{align*}
To further simplify this expression, we rewrite the argument of the additive character $\theta$ and perform another change of variables. Factoring yields
\begin{align*}
&y_1\cdots y_n +x_2+\ldots+ x_n+y_1x_2+y_2x_3+\ldots+ (-1)^ny_{n-1}x_n\lambda^d+ y_n \\
&\hspace{.5in}=y_n(1+ y_1\cdots y_{n-1})+x_2(1+y_1)+x_3(1+y_2)+\ldots +x_n(1+(-1)^ny_{n-1}\lambda^d).
\end{align*}
If any of $y_1, \ldots, y_{n-2}=-1, y_1\cdots y_{n-1}=-1,$ or $y_{n-1}=(-1)^{n+1}/\lambda^d$, then the entire sum is 0. To see this, note that if, for example, $y_1=-1$, then the expression becomes 
\begin{align*}
&T^{m'+a_1}(-1)\sum_{x_i, y_i}T^{a_3+b_1}(x_3)\cdots T^{a_n+b_{n-1}}(x_n)T^{b_1+a_1}(y_1)\cdots T^{b_n+a_1}(y_n)\\
&\hspace{.5in}\cdot\theta(y_n(1+ y_1\cdots y_{n-1})+x_3(1+y_2)+\ldots +x_n(1+(-1)^ny_{n-1}\lambda^d))\cdot\sum_{x_2} T^{a_2+b_1}(x_2),
\end{align*}
and $\sum_{x_2} T^{a_2+b_1}(x_2)=0$ when $a_2+b_1\not=0$.\\

For all other values, we perform the following changes of variables.
$$y_n\rightarrow y_n/(1+ y_1\cdots y_{n-1}), \hspace{.1in} x_2\rightarrow x_2/(1+y_1),\hspace{.1in} \ldots ,\hspace{.1in} x_{n}\rightarrow x_n/(1+ (-1)^ny_{n-1}\lambda^d).$$
This yields the following expression.
\begin{align*}
&T^{m'+a_1}(-1)\sum_{x_i, y_i}T^{a_2+b_1}(x_2)\cdots T^{a_n+b_{n-1}}(x_n)T^{b_n+a_1}(y_n)\cdot\theta(y_n+ x_2+ \ldots +x_n)\\
&\hspace{.5in}T^{b_1+a_1}(y_1)T^{-(a_2+b_1)}(1+y_1)T^{b_2+a_1}(y_2)T^{-(a_3+b_2)}(1+y_2)\\
&\hspace{.5in}\cdots T^{b_{n-1}+a_1}(y_{n-1})T^{-(a_n+b_{n-1})}(1+(-1)^ny_{n-1}\lambda^d)T^{-(b_n+a_1)}(1+ y_1\cdots y_{n-1}).
\end{align*}

Note that the summand equals 0 whenever $y_1, \ldots, y_{n-2}=-1$, $y_1\cdots y_{n-1}=-1,$ or ${y_{n-1}=(-1)^{n+1}/\lambda^d}$, so we can include those values back in the sum. The first part of this summand becomes a product of Gauss sums:

\begin{align*}
G&:=    \sum_{x_i, y_n}T^{a_2+b_1}(x_2)\cdots T^{a_n+b_{n-1}}(x_n)T^{b_n+a_1}(y_n)\cdot\theta(y_n+ x_2+ \ldots +x_n)\\
&=g(T^{a_2+b_1})\cdots g(T^{a_n+b_{n-1}})g(T^{b_n+a_1}).
\end{align*}
So, we write our expression as
\begin{align*}
&T^{m'+a_1}(-1)G\sum_{y_i}T^{b_1+a_1}(y_1)T^{-(a_2+b_1)}(1+y_1)T^{b_2+a_1}(y_2)T^{-(a_3+b_2)}(1+y_2)\\
&\hspace{.5in}\cdots T^{b_{n-1}+a_1}(y_{n-1})T^{-(a_n+b_{n-1})}(1+(-1)^ny_{n-1}\lambda^d)T^{-(b_n+a_1)}(1+ y_1\cdots y_{n-1}).
\end{align*}

In order to get the remaining expression to match Equation \ref{eqn:FnDefinition} We need to perform more changes of variables. First, let $y_{n-1}\rightarrow (-1)^ny_{n-1}/\lambda^d$ to get 
\begin{align*}
&T^{m'+a_1}(-1)G\sum_{y_i}T^{b_1+a_1}(y_1)T^{-(a_2+b_1)}(1+y_1)T^{b_2+a_1}(y_2)T^{-(a_3+b_2)}(1+y_2)\\
&\hspace{.5in}\cdots T^{b_{n-1}+a_1}(y_{n-1})T^{-(a_n+b_{n-1})}(1+y_{n-1})T^{-(b_n+a_1)}(1+ (-1)^ny_1\cdots y_{n-1}\lambda^{-d}).
\end{align*}

We now let $y_i\rightarrow -y_i$ for all $i$ and, noting that $(-1)^{n-1}(-1)^n=-1$, get
\begin{align*}
&T^{m''}(-1) G\sum_{y_i}T^{b_1+a_1}(y_1)T^{-(a_2+b_1)}(1-y_1)T^{b_2+a_1}(y_2)T^{-(a_3+b_2)}(1-y_2)\\
&\hspace{.5in}\cdots T^{b_{n-1}+a_1}(y_{n-1})T^{-(a_n+b_{n-1})}(1-y_{n-1})T^{-(b_n+a_1)}(1 -y_1\cdots y_{n-1}\lambda^{-d}),
\end{align*}
where
\begin{align*}
    m''&=m'+a_1+(b_1+a_1)+(b_2+a_1)+\ldots+(b_{n-1}+a_1)\\
     &=na_1.
\end{align*}

Finally, applying Equation \ref{eqn:FnDefinition} to this yields
\begin{align*}
    T^{m}(-1)G\cdot q^{n-1}{}_{n}F_{n-1}\left(\left.\begin{array}{cccc}
                T^{b_n+a_1}& T^{b_1+a_1} &\ldots &T^{b_{n-1}+a_1}\\
		{} &T^{a_1-a_2}&\ldots&T^{a_1-a_n}              \end{array}\right|\lambda^d\right)_q,
\end{align*}
where 
\begin{align*}
 m&=na_1-((b_1+a_1)+\ldots+(b_{n-1}+a_1)+(a_1-a_2)+\ldots+(a_1-a_n))   \\
  &=\sum_1^n a_i -\sum_1^{n-1} (b_i +a_1).
\end{align*}

\end{proof}

\section{Hypergeometric Point Count Formula}\label{sec:GeneralHypFormula}

Our main interest in Theorem \ref{thm:GaussSumIDFn} is that it can used to simplify Gauss sum expressions in the point count formula for Dwork hypersurfaces. We start by recalling Koblitz's formula given in \cite{KoblitzHypersurface}.\\

Let $W$ be the set of all $d$-tuples $w=(w_1,\ldots,w_d)$ satisfying $0\leq w_i<d$ and $\sum w_i\equiv 0 \pmod d$. We denote the points on the diagonal hypersurface
$$x_1^d+\ldots+x_d^d=0$$
by $N_q(0):=\sum N_{q}(0,w)$, where
$$ N_{q}(0,w)=
 \begin{cases}
  0 &\text{if some but not all } w_i=0,\\
  \frac{q^{d-1}-1}{q-1} &\text{if all } w_i=0,\\
  -\frac1q J\left(T^{\tfrac{w_1}{d}},\ldots,T^{\tfrac{w_d}{d}}\right) &\text{if all } w_i\not=0.\\
 \end{cases}
$$

Letting $W^{**}$ be set of all $d-$tuples where no $w_i=0$, we can write
\begin{equation}\label{eqn:Nq0}
N_{q}(0)=\frac{q^{d-1}-1}{q-1} + \frac1q\sum_{w\in W^{**}} \prod_i g(T^{w_it}).
\end{equation}

As in \cite{GoodsonDwork2017}, we consider cosets of $W$ with respect to the equivalence relation $\sim$ on $W$ defined by $w\sim w'$ if $w-w'$ is a multiple of $(1,\ldots,1)$; we denote this set $W/\sim$ by $W^*$. In the case where $d=4$, we showed in \cite{GoodsonDwork2017} that there were three cosets and their permutations. For general $d$, we should expect many more cosets. We discuss the format of these cosets in Section \ref{sec:TypesofTerms}.\\

Koblitz's formula in this general case is as follows.
$$\#X_{\lambda}^d(\mathbb F_q)=N_q(0)+\frac1{q-1}\sum\frac{\prod_{i=1}^dg\left(T^{w_it+j}\right)}{g(T^{dj})}T^{dj}(d\lambda)$$
where the sum is taken over $j\in\{0,\ldots,q-2\}$ and $w\in W^*$.\\

In \cite[Theorem 8.1]{GoodsonDwork2017}, we gave a partial breakdown of Koblitz's formula and showed that the point count could be given at least partially in terms of hypergeometric functions. We now work to improve on this theorem by rewriting the final summand that appears in the formula.

\begin{prop}\label{prop:KoblitzBreakDown}
Let $q$ be a prime power such that $q\equiv 1\pmod d$, $t=\frac{q-1}{d}$, and $T$ be a generator for $\widehat{\mathbb F_q^{\times}}$. Then, for each $w\in W^*$ and for $\lambda^d\not=1$,
$$\frac{1}{q-1}\sum_{j=0}^{q-2}\frac{\prod_{i=1}^dg\left(T^{w_it+j}\right)}{g(T^{dj})}T^{dj}(d\lambda)$$
can be expressed as a finite field hypergeometric function plus a Gauss sum expression.
\end{prop}
\begin{rem*}
In Section \ref{sec:GaussSumCancel} we will show that the extra Gauss sum expression cancels with the one that appears in the formula for $N_q(0)$. Thus, we obtain a point count formula that is given solely by $(q^{d-1}-1)/(q-1)$ plus a sum of hypergeometric functions.
\end{rem*}

\begin{proof}
We start by assuming that $w_i=0$ for at least one $i$. Note that this will not restrict our use of the proposition since, for any coset $[w]$, we can always choose a representative with at least one 0. We use the specialization of the Hasse-Davenport formula given in Corollary \ref{cor:HasseDavenportGeneral}, and then cancel all common factors, to obtain

\begin{align*}
    \sum_{j=0}^{q-2}\frac{g(T^{w_1t+j})\cdots g(T^{w_dt+j})}{g(T^{dj})}T^{dj}(d\lambda) &=\prod_{k=1}^{d-1}g(T^{kt})\sum_{j=0}^{q-2}\frac{g(T^{w_1t+j})\cdots g(T^{w_dt+j})}{g(T^{j})g(T^{t+j})\cdots g(T^{(d-1)t +j})}T^{dj}(\lambda)\\
    &=\prod_{k=1}^{d-1}g(T^{kt})\sum_{j=0}^{q-2}\frac{g(T^{a_1t+j})\cdots g(T^{a_nt+j})}{g(T^{b_1t+j})\cdots g(T^{b_nt +j})}T^{dj}(\lambda),
\end{align*}
where $n<d$. When $d$ is odd, 
\begin{align*}
    \prod_{k=1}^{d-1}g(T^{kt})&= q^{\frac{d-1}{2}}T^\alpha(-1),
\end{align*}
where $\alpha = \sum_1^{(d-1)/2)}it=\frac{(d-1)(d+1)}{8}t$. On the other hand, when $d$ is even, the Gauss sum factors do not perfectly pair up. In this case we get
\begin{align*}
    \prod_{k=1}^{d-1}g(T^{kt})&= q^{\frac{d-2}{2}}g(T^{\frac{d}{2}t})T^{\alpha'}(-1),
\end{align*}
where $\alpha' = \sum_1^{(d-2)/2)}it=\frac{(d-2)(d)}{8}t$. In both cases we will denote this quantity by $G_d$ with the understanding that its value depends on the parity of $d$.\\

We can order terms in the Gauss sum expression so that $0\leq a_1\leq \ldots \leq a_n$ and $0<b_1<\ldots <b_n$. The inequalities on the $b_i$ are strict because the factors in the denominator were distinct and, since $w_i=0$ for some $i$, $b_1\not =0$. Note that $b_i\not= a_j$ for all $i,j$ because otherwise the corresponding factors would have canceled.\\

Next, we would like to rewrite this using the relation
$$g(T^{b_it+j})g(T^{-b_it-j})=T^{b_it+j}(-1)q,$$
but we must first remove all $j=(d-b_i)t$ from the summand since the above relation holds only when $T^{b_it+j}\not= \epsilon$. Note that for each $j=(d-b_i)t$, we have
\begin{align*}
    \frac{g(T^{a_1t+j})\cdots g(T^{a_nt+j})}{g(T^{b_1t+j})\cdots g(T^{b_nt +j})}T^{dj}(\lambda)&=\frac{g(T^{(a_1-b_i)t})\cdots g(T^{(a_n-b_i)t})}{g(T^{(b_1-b_i)t})\cdots g(T^{(b_i-b_i)t})\cdots g(T^{(b_n-b_i)t}) }\\ 
    &=-\frac{g(T^{(a_1-b_i)t})\cdots g(T^{(a_n-b_i)t})}{g(T^{(b_1-b_i)t})\cdots \widehat{g(T^{(b_i-b_i)t})}\cdots g(T^{(b_n-b_i)t}) } \\
    &=-\frac{T^{m'}(-1)}{q^{n-1}}g(T^{(a_1-b_i)t})\cdots g(T^{(a_n-b_i)t}) \\ 
    &\hspace{.5in}\cdot g(T^{(-b_1+b_i)t})\cdots \widehat{g(T^{(-b_i+b_i)t})}\cdots g(T^{(-b_n+b_i)t}), 
\end{align*}
where $m'=\sum_{k\not= i}b_kt -(n-1)b_it=\sum_{k=1}^n b_kt -nb_it$.\\

For the remaining terms we can write
\begin{align*}
    \sum_{j\not=-b_it}\frac{g(T^{a_1t+j})\cdots g(T^{a_nt+j})}{g(T^{b_1t+j})\cdots g(T^{b_nt +j})}T^{dj}(\lambda)&=\sum_{j\not=-b_it}\frac{T^{m''}(-1)}{q^{n}}g(T^{a_1t+j})\cdots g(T^{a_nt+j})\\ 
    &\hspace{.5in}\cdot g(T^{-b_1t-j})\cdots g(T^{-b_nt-j})T^{dj}(\lambda), 
\end{align*}
where $m''=\sum_{k=1}^n b_kt +nj$.\\

Note that if $j=-b_it$, then $m''=\sum_{k=1}^n b_kt +nj= \sum_{k=1}^n b_kt -nb_it= m'$ and
\begin{align*}
    \frac{T^{m''}(-1)}{q^{n}}&g(T^{a_1t+j})\cdots g(T^{a_nt+j}) g(T^{-b_1t-j})\cdots g(T^{-b_nt-j})T^{dj}(d\lambda)\\
    &=\frac{T^{m'}(-1)}{q^{n}}g(T^{(a_1-b_i)t})\cdots g(T^{(a_n-b_i)t})g(T^{(-b_1+b_i)t})\cdots g(T^{(-b_i+b_i)t})\cdots g(T^{(-b_n+b_i)t})\\
    &=-\frac{T^{m'}(-1)}{q^{n}}g(T^{(a_1-b_i)t})\cdots g(T^{(a_n-b_i)t}) g(T^{(-b_1+b_i)t})\cdots \widehat{g(T^{(-b_i+b_i)t})}\cdots g(T^{(-b_n+b_i)t}).
\end{align*}

Putting this all together yields the following.
\begin{align*}
    \sum_{j=0}^{q-2}\frac{g(T^{a_1t+j})\cdots g(T^{a_nt+j})}{g(T^{b_1t+j})\cdots g(T^{b_nt +j})}T^{dj}(\lambda)&= \sum_{j=0}^{q-2}\frac{T^{m''}(-1)}{q^{n}}\left(\prod_{i=1}^n g(T^{a_it+j})\cdot \prod_{i=1}^n g(T^{-b_it-j})\right)T^{dj}(\lambda)\\
    &\hspace{.4in}+\left(-\sum_{i=1}^n\frac{T^{m'}(-1)}{q^{n-1}}\prod_{k=1}^ng(T^{(a_k-b_i)t})\cdot\prod_{k\not=i}^n g(T^{(-b_k+b_i)t})\right)\\
     &\hspace{.4in}-\left(-\sum_{i=1}^n\frac{T^{m'}(-1)}{q^{n}}\prod_{k=1}^ng(T^{(a_k-b_i)t})\cdot\prod_{k\not=i}^n g(T^{(-b_k+b_i)t})\right)\\
     &= \sum_{j=0}^{q-2}\frac{T^{m''}(-1)}{q^{n}}\left(\prod_{i=1}^n g(T^{a_it+j})\cdot \prod_{i=1}^n g(T^{-b_it-j})\right)T^{dj}(\lambda)\\
    &\hspace{.4in}-\frac{q-1}{q^n}\sum_{i=1}^nT^{m'}(-1)\prod_{k}g(T^{(a_k-b_i)t})\cdot\prod_{k\not=i} g(T^{(-b_k+b_i)t})\\
     &=\frac{1}{q^{n}} \sum_{j=0}^{q-2}\left(\prod_{i=1}^n g(T^{a_it+j})\cdot \prod_{i=1}^n T^{b_it+j}(-1)g(T^{-b_it-j})\right)T^{j}(\lambda^d)\\
    &\hspace{.4in}-\frac{q-1}{q^{n}}\sum_{i=1}^n\prod_{k=1}^ng(T^{(a_k-b_i)t})\cdot\prod_{k\not=i} T^{(b_k-b_i)t}(-1)g(T^{(-b_k+b_i)t}).
\end{align*}

Theorem \ref{thm:GaussSumIDFn} tells us that 
\begin{align*}
    &\sum_{j=0}^{q-2}\left(\prod_{i=1}^n g(T^{a_it+j})\cdot \prod_{i=1}^n T^{b_it+j}(-1)g(T^{-b_it-j})\right)T^{j}(\lambda^d)\\
    &\hspace{.6in}=T^{m}(-1)G\cdot q^{n-1}(q-1){}_{n}F_{n-1}\left(\left.\begin{array}{cccc}
                T^{(a_1-b_n)t}& T^{(a_1-b_1)t} &\ldots &T^{(a_1-b_{n-1})t}\\
		{} &T^{(a_1-a_2)t}&\ldots&T^{(a_1-a_n)t}              \end{array}\right|\lambda^d\right)_q,
\end{align*}
where $m=\sum_1^n a_it -\sum_1^{n-1} (a_1-b_i)t$ and $G=g(T^{(a_2-b_1)t})\cdots g(T^{(a_n-b_{n-1})t})g(T^{(a_1-b_n)t})$.\\

Putting all of this together yields
\begin{align*}
 \frac{1}{q-1}\sum_{j=0}^{q-2}&\frac{\prod_{i=1}^dg\left(T^{w_it+j}\right)}{g(T^{dj})}T^{dj}(d\lambda) \\
 &= \frac{G_d}{q-1}\left(\frac{1}{q^{n}} \sum_{j=0}^{q-2}\left(\prod_{i=1}^n g(T^{a_it+j})\cdot \prod_{i=1}^n T^{b_it+j}(-1)g(T^{-b_it-j})\right)T^{j}(\lambda^d)\right.\\
    &\hspace{.5in}\left.-\frac{q-1}{q^{n}}\sum_{i=1}^n\prod_{k=1}^ng(T^{(a_k-b_i)t})\cdot\prod_{k\not=i} T^{(b_k-b_i)t}(-1)g(T^{(-b_k+b_i)t}) \right)\\
    &=\frac{T^{m}(-1)G\cdot G_d}{q}{}_{n}F_{n-1}\left(\left.\begin{array}{cccc}
                T^{(a_1-b_n)t}& T^{(a_1-b_1)t} &\ldots &T^{(a_1-b_{n-1})t}\\
		{} &T^{(a_1-a_2)t}&\ldots&T^{(a_1-a_n)t}              \end{array}\right|\lambda^d\right)_q\\
    &\hspace{.5in}-\frac{G_d}{q^{n}}\sum_{i=1}^n\prod_{k=1}^ng(T^{(a_k-b_i)t})\cdot\prod_{k\not=i}^n T^{(b_k-b_i)t}(-1)g(T^{(-b_k+b_i)t}) \\
\end{align*}

We can clean this up slightly by noting that when $d$ is odd
\begin{align*}
  \frac{T^{m}(-1) G_d}{q} &=T^{m+(d-1)(d+1)/8}(-1)\cdot q^{(d-1)/2-1},
\end{align*}
and when $d$ is even
\begin{align*}
  \frac{T^{m}(-1) G_d}{q} &=T^{m+(d-2)(d)/8}(-1)\cdot q^{(d-2)/2-1}\cdot g(T^{\frac{d}{2}t}).
\end{align*}

In both cases we will denote this quantity by $G_d'$ with the understanding that its value partially depends on the parity of $d$. Hence, we can write that
\begin{align*}
 \frac{1}{q-1}\sum_{j=0}^{q-2}\frac{\prod_{i=1}^dg\left(T^{w_it+j}\right)}{g(T^{dj})}&T^{dj}(d\lambda)\\
    &=G\cdot G_d'\cdot{}_{n}F_{n-1}\left(\left.\begin{array}{cccc}
                T^{(a_1-b_n)t}& T^{(a_1-b_1)t} &\ldots &T^{(a_1-b_{n-1})t}\\
		{} &T^{(a_1-a_2)t}&\ldots&T^{(a_1-a_n)t}              \end{array}\right|\lambda^d\right)_q\\
    &\hspace{.5in}-\frac{G_d}{q^{n}}\sum_{i=1}^n\prod_{k=1}^ng(T^{(a_k-b_i)t})\cdot\prod_{k\not=i}^n T^{(b_k-b_i)t}(-1)g(T^{(-b_k+b_i)t}). \\
\end{align*}

\end{proof}

\subsection{Canceling Gauss Sum Expressions}\label{sec:GaussSumCancel}
At the end of the proof of Proposition \ref{prop:KoblitzBreakDown}, we obtain a finite field hypergeometric function minus the Gauss sum expression
\begin{equation}\label{eqn:ExtraGauss}
\frac{G_d}{q^{n}}\sum_{i=1}^n\prod_{k=1}^ng(T^{(a_k-b_i)t})\cdot\prod_{k\not=i}^n T^{(b_k-b_i)t}(-1)g(T^{(-b_k+b_i)t}),
\end{equation}
where 
$$G_d=\begin{cases}
q^{\frac{d-1}{2}}T^{(d-1)(d+1)t/8}(-1) & \text{if } d \text{ is odd},\\
q^{\frac{d-2}{2}}g(T^{\frac{d}{2}t})T^{(d-2)(d)t/8}(-1) & \text{if } d \text{ is even}.
\end{cases}$$

In Equation \ref{eqn:Nq0} we defined the quantity $N_q(0)$, which makes up part of the overall point count for the $(d-2)$-dimensional Dwork hypersurface. The quantity is given by  $(q^{d-1}-1)/(q-1)$ plus
\begin{equation}\label{eqn:W**}
\frac1q\sum_{w\in W^{**}} \prod_i g(T^{w_it}),
\end{equation}
where $W$ is the set of all $d$-tuples $w=(w_1,\ldots,w_d)$ satisfying $0\leq w_i <d$ and $\sum w_i \equiv 0 \pmod d$ and $W^{**}$ is the set of all $w=(w_1,\ldots,w_d)$ with no $w_i=0$. Our aim in this section is to show that the Gauss sum expression in Eqaution \ref{eqn:W**} will negate all of the terms of the form in Equation \ref{eqn:ExtraGauss}. In doing so, we prove that the number of points on any Dwork hypersurface can be expressed as a sum of finite field hypergeometric functions plus the quantity $(q^{d-1}-1)/(q-1)$. \\

By definition, the following are true of the numbers appearing in Equation \ref{eqn:ExtraGauss}.
\begin{align}\label{RulesForExponents}
    &b_j\not=0, \text{ for all } j \nonumber \\
    &a_k\not=b_j, \text{ for all } k,j, \text{ and} \\
    &\text{The } b_j \text{ are distinct}.\nonumber
\end{align}

This implies that no factor in the product equals $g(T^0)$. Furthermore, the sum of the exponents in each of the Gauss sums is congruent to $0 \pmod {q-1}$. Fixing $i$, we see that
\begin{align*}
    \sum_{k=1}^n(a_k-b_i)t + \sum_{k=1, k\not=i}^n(-b_k+b_i)t&= t\left(\sum_{k=1}^n a_k -nb_i -\sum_{k=1}^n b_k +nb_i\right)\\
    &=t\left(\sum_{k=1}^n a_k -\sum_{k=1}^n b_k\right).
\end{align*}
We can write that
\begin{equation*}
    \sum_{k=1}^n b_k = \sum_{j=0}^{d-1} j - \sum a_j'= \frac{(d-1)d}{2} - \sum a_j',
\end{equation*}
where the $a_j'$ are the common factors that were canceled from the numerator and denominator in the original Gauss sum expression at the beginning of the proof of Proposition \ref{prop:KoblitzBreakDown}. Thus, the sum of the exponents is
\begin{align*}
    t\left(\sum_{k=1}^n a_k -\sum_{k=1}^n b_k\right)&=t\left(\sum_{k=1}^n a_k + \sum a_j' -\frac{(d-1)d}{2}\right) \\
    &= \sum_{j=1}^d w_jt -\frac{(d-1)dt}{2}.
\end{align*}

We will now split into two cases: $d$ odd and $d$ even. Our work with $d$ odd leads to a proof of Theorem \ref{thm:DworkHypOdd}. Our work with $d$ even gives progress toward a proof of Conjecture \ref{conjec:DworkHypEven}.

\subsubsection{Proof of Theorem \ref{thm:DworkHypOdd}}\label{subsec:d_odd}

Note that, by definition, $\sum_{j=1}^d w_j\equiv 0 \pmod d$. When $d$ is odd, $(d-1)$ is even and, hence, $\sum_{j=1}^d w_j -\frac{(d-1)d}{2}$ is a multiple of $d$. Hence, the sum of the exponents will be congruent to $dt$, which is congruent to $0 \pmod{q-1}$. Note that this is also true of the expressions in Equation \ref{eqn:W**}.\\

The Gauss sum expression in Equation \ref{eqn:ExtraGauss} has $(2n-1)$ factors. If $2n-1 =d$ then
$$\prod_{k=1}^ng(T^{(a_k-b_i)t})\cdot\prod_{k\not=i}^n g(T^{(-b_k+b_i)t}) = \prod_{i=1}^d g(T^{w_it}),$$
for some $w=(w_1,\ldots,w_d)$ in $W^{**}$.\\ 

If instead we have $2n-1<d$, then there exist $v_1,\ldots v_l$, with $0<v_j<d$, that complete the product, i.e. 
$$\prod_{k=1}^ng(T^{(a_k-b_i)t})\cdot\prod_{k\not=i}^n g(T^{(-b_k+b_i)t}) \cdot \prod_{k=1}^l g(T^{v_k})g(T^{d-v_k})= \prod_{i=1}^d g(T^{w_it}),$$
for some $w=(w_1,\ldots,w_d)$ in $W^{**}$ by the properties listed in (\ref{RulesForExponents}).\\

Now suppose $2n-1>d$, i.e. there are too many factors in the Gauss sum expression of Equation \ref{eqn:ExtraGauss}. We will show that we can find enough factor pairs of the form $g(T^{vt})g(T^{(d-v)t})=qT^{vt}(-1)$ to reduce the number of factors to $d$.\\

Since $2n-1>d$, we can write that $n-1=\frac{d+m}{2}$, for some $m\geq1$. Recall that the exponents in the expression
$$\prod_{k\not=i}^n g(T^{(-b_k+b_i)t})$$
are distinct because each $b_j$ is distinct. It is possible to have a maximum of $\frac{d-1}{2}$ numbers in the list of these exponents without having any pairs of the form $i, d-i$. These pairs will yield $g(T^{it})g(T^{(d-i)t})=qT^i(-1)$ in our Gauss sum expression, which reduces the overall number of factors. There will be $\frac{d+m}{2}-\frac{d-1}{2}= \frac{m+1}{2}$ pairs. This will leave us with $d$ factors in the Gauss sum expression since
$$2n-1- 2\left(\frac{d+m}{2}-\frac{d-1}{2}\right)=(d+m+2)-1 -(m+1)=d.$$ 

Thus,
$$\prod_{k=1}^ng(T^{(a_k-b_i)t})\cdot{\prod_k}' g(T^{(-b_k+b_i)t})= \prod_{i=1}^d g(T^{w_it}),$$
for some $w=(w_1,\ldots,w_d)$ in $W^{**}$ and for some subset of $\{-b_k+b_i\}_{k=1}^{n}$.\\

We now compare the size of $W^{**}$ to the number of expressions of the form in Equation \ref{eqn:ExtraGauss} that appear in the point count formula. Let $[w]\in W^*=W/\sim$ and let $N_w=\#\{w'\in[w]: w' \text{ contains no zeros}\}$. Note that $N_w$ corresponds to the number of distinct numbers in the $d$-tuple $w=(w_1,\ldots,w_d)$ in the following way. If $s$ is the number of distinct numbers in $w$, then $N_w=d-s$. For example, if $d=5$ and $w=(0,0,0,2,3)$, then $N=5-3=2$ and these coset elements are
$$(1,1,1,3,4) \text{ and } (4,4,4,1,2).$$
All other elements in the coset $[(0,0,0,2,3)]$ contain at least one zero. Furthermore, $\sum_{[w]}N_w$ gives the total number of elements in $W^{**}$. \\

We now consider the number of expressions of the form in Equation \ref{eqn:ExtraGauss} that we obtain. For each coset representative $[w]$, we obtain the sum
\begin{equation*}
    \frac{G_d}{q^{n}}\sum_{i=1}^n\prod_{k=1}^ng(T^{(a_k-b_i)t})\cdot\prod_{k\not=i}^n T^{(b_k-b_i)t}(-1)g(T^{(-b_k+b_i)t}),
\end{equation*}
where $n$ corresponds to the number of terms left after canceling common factors in the numerator and denominator. Hence, $n$ also equals $d-s$. Adding these values for each $[w]$ yields the total number of expressions of the form in Equation \ref{eqn:ExtraGauss}.\\

Thus, what we have shown is that there is a matching of expressions in our two calculations and that the size of $W^{**}$ equals the number of expressions of the form in Equation \ref{eqn:ExtraGauss} that appear on the point count formula. Our final task is to show that the coefficients of the matching Gauss sum expressions are equal. This entails checking that the powers of $q$ match and that the character evaluation of $-1$ in Equation \ref{eqn:ExtraGauss} matches that of Equation \ref{eqn:W**}.\\

First note that in Equation \ref{eqn:ExtraGauss}, the exponent for character evaluation $T(-1)$ is
$$\frac{(d-1)(d+1)t}{8} +\sum_{k=1}^n b_kt-nb_it.$$

Since $T^t$ is a character of order $d$ and $d$ is odd, we have
$$T^t(-1)=T^{dt}(-1)=1.$$

We now examine the power of $q$ in Equation \ref{eqn:ExtraGauss}. To start, we have 
$$\frac{q^{(d-1)/2}}{q^n}\cdot \left[(2n-1) \text{ Gauss sums}\right].$$
Note that if $2n-1=d$, i.e. $n=\frac{d+1}{2}$, then we are left with a coefficient of $\frac1q$. Furthermore, we showed above that we can always make this a product of exactly $d$ Gauss sums by canceling out (if $2n-1>d$) or multiplying (if $2n-1<d$) pairs of the form $g(T^{it})g(T^{(d-i)t})$. For each of these pairs, we introduce (if $2n-1>d$) or remove (if $2n-1<d$) a factor of $q$. Regardless of which case we are in, the exponent of $q$ becomes
$$\frac{d-1}{2} - n +\frac{2n-1-d}{2}= -1.$$

Thus, Equation \ref{eqn:ExtraGauss} can be rewritten as
\begin{align*}
\frac{G_d}{q^{n}}\sum_{i=1}^n\prod_{k=1}^ng(T^{(a_k-b_i)t})&\cdot\prod_{k\not=i}^n T^{(b_k-b_i)t}(-1)g(T^{(-b_k+b_i)t})\\
&=\prod_{k=1}^ng(T^{(a_k-b_i)t})\cdot\prod_{k\not=i}^n g(T^{(-b_k+b_i)t}) = \frac{1}{q}\prod_{j=1}^d g(T^{w_it})\end{align*}
where $w=(w_1,\ldots,w_d)$ has no $w_j=0$ and satisfies $\sum w_j\equiv 0 \pmod d$. Since the coefficient of each of these expressions in the point count formula of Proposition \ref{prop:KoblitzBreakDown} is $-1$, the sum of all of these terms will negate the extra Gauss sum expression shown in Equation \ref{eqn:W**} that appears in the point count formula. Hence, the number of points on any Dwork hypersurface can be expressed as a sum of finite field hypergeometric functions plus the quantity $(q^{d-1}-1)/(q-1)$.

\subsubsection{Progress on Conjecture \ref{conjec:DworkHypEven}.}\label{subsec:d_even}
The case where $d$ is even is similar, though slightly more complicated. The main obstruction to a complete result is that, unlike in the case when $d$ is odd, $T^t(-1)$ does not necessarily equal 1. \\

Recall that when $d$ is even, the Gauss sum expression in Equation \ref{eqn:ExtraGauss} has a factor of $g(T^{\frac{d}{2}t})$. Thus, the exponent sum we consider in this case is
\begin{align*}
    \sum_{j=1}^d w_j -\frac{(d-1)d}{2} +\frac{d}{2}&= \sum_{j=1}^d w_j -\frac{(d-2)d}{2}\\
    &\equiv 0 \pmod d.
\end{align*}

Hence, the sum of the exponents will be congruent to $dt$, which is congruent to 0 modulo ${q-1}$. Note that this is also true of the expressions in Equation \ref{eqn:W**}.\\

The Gauss sum expression in Equation \ref{eqn:ExtraGauss} has $2n$ factors when $d$ is even, one of which is always $g(T^{\frac{d}{2}t})$. If $2n=d$ then
$$g(T^{\frac{d}{2}t})\prod_{k=1}^ng(T^{(a_k-b_i)t})\cdot\prod_{k\not=i}^n g(T^{(-b_k+b_i)t}) = \prod_{i=1}^d g(T^{w_it}),$$
for some $w=(w_1,\ldots,w_d)$ in $W^{**}$.\\ 

If instead we have $2n<d$, then there exist $v_1,\ldots, v_l$, with $0<v_j<d$, that complete the product, i.e. 
$$g(T^{\frac{d}{2}t})\prod_{k=1}^ng(T^{(a_k-b_i)t})\cdot\prod_{k\not=i}^n g(T^{(-b_k+b_i)t}) \cdot \prod_{k=1}^l g(T^{v_k})g(T^{d-v_k})= \prod_{i=1}^d g(T^{w_it}),$$
for some $w=(w_1,\ldots,w_d)$ in $W^{**}$.\\

Now suppose $2n>d$, i.e. there are too many factors in the Gauss sum expression of Equation \ref{eqn:ExtraGauss}. We will show that we can find enough factor pairs of the form $g(T^{vt})g(T^{(d-v)t})=qT^{vt}(-1)$ to reduce the number of factors to $d$.\\

Since $2n>d$, we can write that $n-1=\frac{d+m}{2}$, for some $m\geq0$. Recall that the exponents in the expression
$$\prod_{k\not=i}^n g(T^{(-b_k+b_i)t})$$
are distinct because each $b_j$ is distinct. It is possible to have $\frac{d-2}{2}$ numbers in the list of exponents without having any pairs of the form $i, d-i$. Since there are more terms than this, there will be $\frac{d+m}{2}-\frac{d-2}{2}= \frac{m+2}{2}$ pairs. This includes the possibility that one of the terms is $\frac{d}{2}$, which yields a Gauss sum expression that cancels with the other $g(T^{\frac{d}{2}t})$ factor that is always there. This will leave us with $d$ factors in the Gauss sum expression since
$$2n- 2\left( \frac{m+2}{2}\right)=(d+m+2) -(m+2)=d.$$ 

Thus,
$$\prod_{k=1}^ng(T^{(a_k-b_i)t})\cdot{\prod_k}' g(T^{(-b_k+b_i)t})= \prod_{i=1}^d g(T^{w_it}),$$
for some $w=(w_1,\ldots,w_d)$ in $W^{**}$ and for some subset of $\{-b_k+b_i\}_{k=1}^{n}$.\\

We now compare the size of $W^{**}$ to the number of expressions of the form in Equation \ref{eqn:ExtraGauss} that appear on the point count formula. Let $[w]\in W^*=W/\sim$. Let $N_w=\#\{w'\in[w]: w' \text{ contains no zeros}\}$. Note that $N_w$ corresponds to the number of distinct numbers in the $d$-tuple $w=(w_1,\ldots,w_d)$ in the following way. If $s$ is the number of distinct numbers in $w$, then $N_w=d-s$. For example, if $d=6$ and $w=(0,0,0,1,2,3)$, then $N=6-4=2$ and these coset elements are
$$(1,1,1,2,3,4) \text{ and } (2,2,2,3,4,5).$$
All other elements in the coset $[(0,0,0,1,2,3)]$ contain at least one zero. Furthermore, $\sum_{[w]}N_w$ gives the total number of elements in $W^{**}$. \\

We now consider the number of expressions of the form in Equation \ref{eqn:ExtraGauss} we obtain. For each coset representative $[w]$, we obtain the sum
\begin{equation*}
    \frac{G_d}{q^{n}}\sum_{i=1}^n\prod_{k=1}^ng(T^{(a_k-b_i)t})\cdot\prod_{k\not=i}^n T^{(b_k-b_i)t}(-1)g(T^{(-b_k+b_i)t}),
\end{equation*}
where $n$ corresponds to the number of terms left after canceling common factors in the numerator and denominator. Hence, $n$ also equals $d-s$. Adding these values for each $[w]$ yields the total number of expressions of the form in Equation \ref{eqn:ExtraGauss}.\\

Thus, what we have shown is that there is a matching of expressions in our two calculations and that the size of $W^{**}$ equals the number of expressions of the form in Equation \ref{eqn:ExtraGauss} that appear on the point count formula. We now examine the power of $q$ in Equation \ref{eqn:ExtraGauss}. To start, we have 
$$\frac{q^{(d-2)/2}}{q^n}\cdot \left[(2n) \text{ Gauss sums}\right],$$
where one of the Gauss sums is always $g(T^{dt/2})$. Note that if $2n=d$, i.e. $n=\frac{d}{2}$, then we are left with a coefficient of $\frac1q$. Furthermore, we showed above that we can always make this a product of exactly $d$ Gauss sums by canceling out (if $2n>d$) or multiplying (if $2n<d$) pairs of the form $g(T^{it})g(T^{(d-i)t})=qT^{it}(-1)$. Regardless of which case we are in, the exponent of $q$ becomes
$$\frac{d-2}{2} - n +\frac{2n-d}{2}= -1.$$

Our final task is to show that the character evaluation of $-1$ in Equation \ref{eqn:ExtraGauss} matches that of Equation \ref{eqn:W**}. In general, it is not the case that when $d$ is even that $T^t(-1)=1$. However, it appears that we can always get around this obstruction in the following way. It seems to be the case that when the Gauss sum expression has an extraneous $T^t(-1)$, in the coefficient that there is a Gauss sum pair $g(T^{it})g(T^{(d-i)t}$, where $i$ and $d-i$ have the same parity. Thus, letting $g(T^{jt})g(T^{(d-j)t})$ be another pair where $j$ and $d-j$ have the same parity, but opposite of that of $i$ and $d-i$, we see that
\begin{align*}
    g(T^{it})g(T^{(d-i)t}&=qT^{it}(-1)\\                                              &=qT^{it}(-1)\cdot\left(\frac{1}{qT^{jt}(-1)}g(T^{jt})g(T^{(d-j)t})\right)\\
                         &=T^t(-1)g(T^{jt})g(T^{(d-j)t}).
\end{align*}

Thus, we can use swaps of this sort precisely when we would like to remove a $T^t(-1)$ in the coefficient of the Gauss sum expression. Proving that this can always be done would lead to a proof of Conjecture \ref{conjec:DworkHypEven}.

\subsection{Types of Hypergeometric Terms}\label{sec:TypesofTerms}
We would like to characterize the types of hypergeometric terms that appear in Theorem \ref{thm:DworkHypOdd} and Conjecture \ref{conjec:DworkHypEven} for a given $d$. Recall that $W=\{{w}=(w_1,\ldots,w_d)\}$, where $\sum w_i\equiv 0 \pmod{d}$. In Koblitz's formula, we sum over all cosets in $W^*=W/\sim$, where $w\sim w'$ if $w-w'=(1,\ldots, 1)$. As a convention, we will choose a coset representative with the maximum number of zeros. Note that this choice is not necessarily unique, i.e. there may be more than one element in a coset with the maximum number. For example, when $d=4$ we have $(0,0,2,2)$ and $(2,2,0,0)$ in the same coset.\\

In what follows, we show when certain types of terms will occur in the point count formulas of Theorem \ref{thm:DworkHypOdd} and Conjecture \ref{conjec:DworkHypEven}. Throughout this section, assume $\lambda\not=0$.

\begin{lem}
The point count formula contains a constant term if and only if $d$ is odd and $\lambda^d=1$.
\end{lem}
\begin{proof}
In order to obtain a constant term when we apply Proposition \ref{prop:KoblitzBreakDown} and Theorem \ref{thm:GaussSumIDFn}, all of the entries in $w$ must be distinct. Hence, we must have (up to permutation) $w=(0,1,\ldots,d-1)$. We find that the sum of the entries in $w$ is
\begin{align*}
    \sum_{i=1}^{d-1} i=\frac{d(d-1)}{2}.
\end{align*}
This sum is a multiple of $d$ if and only if $d$ is odd. Thus, the point count formula contains a constant term if and only if $d$ is odd. Furthermore, by counting permutations we see that there will be $(d-1)!$ of these terms and they are all of the form
$$q^{(d-1)/2}\delta(1-\lambda^d), $$
where $\delta(x)=1$ if $x=0$ and $\delta(x)=0$ otherwise. The proof of Theorem \ref{thm:ThreefoldPointCount} in Section \ref{sec:DworkThreefold} describes in greater detail why this is true.
\end{proof}

\begin{lem}
The point count formula contains a ${}_1F_0$ hypergeometric term if and only if $d$ is even. Furthermore, when $d$ is even, there are $\frac{d!}{2}$ such terms.
\end{lem}
\begin{proof}
In order to obtain a ${}_1F_0$ term when we apply Proposition \ref{prop:KoblitzBreakDown} and Theorem \ref{thm:GaussSumIDFn}, we must have (up to permutation) $w=(0,0,w_1,\ldots, w_{d-2})$, where $0<w_1<\ldots<w_{d-2}\leq d-1$. Thus, the set $\{w_1,\ldots, w_{d-2}\}$ contains all but one element of $\{1, \ldots, d-1\}$. Supposing this missing element is $j$, we find that the sum of the entries in $w$ is
\begin{align*}
    0+0+w_1+\ldots+w_{d-2}&=\sum_{i=1}^{d-1} i-j\\
                          &=\frac{d(d-1)}{2}-j.
\end{align*}
Recall that, by definition, this sum is congruent to $0\pmod d$. When $d$ is odd, $\frac{d(d-1)}{2}$ is a multiple of $d$. However, since $j\in\{1,\ldots,d-1\}$, $j$ is not divisible by $d$. Thus, it is not possible to have a $w$ of this form in the set $W^*$, so there will not be a ${}_1F_0$ term in the point count formula.\\

However, when $d$ is even, neither $j$ nor $\frac{d(d-1)}{2}$ is divisible by $d$. In this case if we let $j=d/2$, we find that
\begin{align*}
    0+0+w_1+\ldots+w_{d-2}&=\sum_{i=1}^{d-1} i -\frac{d}{2}\\
                          &=\frac{d(d-1)}{2}-\frac{d}{2}\\
                          &=\frac{d(d-2)}{2},
\end{align*}
which is divisible by $d$. Written as above, this corresponds to the element $$w=(0,0,1,\ldots,{d/2 -1},{d/2+1},\ldots, {d-2}).$$ By counting permutations of this, we see that there will be $\frac{d!}{2}$ elements in $W^*$ that yield a ${}_1F_0$ term. In fact, by using the same proof techniques that were used in the proof of Proposition 4.5 in \cite{GoodsonDwork2017}, one can show that these terms are all of the form
$$qT^t(-1)T^{\frac{d}{2}t}(1-\lambda^d)=qT^t(-1){}_{1}F_{0}\left(\left.\begin{array}{c}
                T^{\frac{d}{2}t}
               \end{array}\right|\lambda^d\right)_q.$$
\end{proof}

In \cite{GoodsonDwork2017}, we proved that there will always be a ${}_{d-1}F_{d-2}$ term in the point count formula. We now consider some other higher order terms, namely ${}_{d-2}F_{d-3}$ and ${}_{d-3}F_{d-4}$.

\begin{lem}
The point count formula contains a ${}_{d-2}F_{d-3}$ term if and only if $d$ is composite.
\end{lem}
\begin{proof}
Recall that a ${}_nF_{n-1}$ hypergeometric function in the point count formula corresponds to a Gauss sum expression that has $n$ factors remaining in the numerator after canceling. Having $d-2$ terms left after canceling means that this Gauss sum expression corresponds to a $w\in W^*$ that has only two distinct numbers in its sequence. Since we are choosing the coset representative $w$ so that it contains the maximum number of zeros, this means we have $w=(0,\ldots, 0 , a, \dots, a)$, where $0<a\leq d-1$. Let $m$ be the number of times $a$ occurs, and note that $m\leq d/2$. An element $w$ of this form is possible only when $m\cdot a \equiv 0 \pmod d$, since otherwise $w$ does not lie in $W^*$.\\

If $d$ is prime, this is never possible. If $d$ is composite, this will be possible whenever $a$ is a divisor of $d$. 
\end{proof}

Finally, we show that there will always be hypergeometric terms whose bottom row entries are all the trivial character. Hypergeometric functions of this form are of particular interest since, in many cases, these are known to be congruent modulo $p$ to classical hypergeometric series (see, for example, \cite{GoodsonDwork2017, Swisher2016}).

\begin{lem}
The point count formula will always have hypergeometric terms whose bottom row entries are all the trivial character. In particular, there will always be ${}_{d-3}F_{d-4}$ hypergeometric terms of this kind. Up to permutation, when $d>3$ is odd there will be $\frac{d-1}{2}$ such terms, and when $d>2$ is even there will be $\frac{d-2}{2}$ such terms.
\end{lem}
\begin{proof}
A ${}_nF_{n-1}$ hypergeometric function whose bottom row entries are all the trivial character in the point count formula corresponds to a Gauss sum expression that has $n$ equal factors remaining in the numerator after canceling
$$a_1=a_2=\ldots=a_n.$$
Note that, in general,  if any factors are left after canceling, at least one of them must be zero since we have chosen $w$ to be the coset representative with the maximum number of zeros. Thus, we must have
$$a_1=a_2=\ldots=a_n=0.$$
This corresponds to the coset representative $w=(0,\ldots, 0 , w_1, \dots, w_{d-(n+1)})$, where the number of zeros is $n+1$ and each $w_i$ is distinct. It is always possible to find a set of distinct $w_i$ satisfying $0<w_i\leq d-1$ and $\sum w_i\equiv 0 \pmod d$. Hence, we will always obtain a hypergeometric term whose bottom row entries are all the trivial character.\\

In particular, it is always possible to find a pair of distinct numbers whose sum is congruent to 0 modulo $d$. When $d>3$ is odd, there will always be $\frac{d-1}{2}$ such pairs:
$$1 +(d-1), 2+(d-2), \ldots, \tfrac{d-1}{2}+\tfrac{d+1}{2}.$$
When $d>2$ is even, there will always be $\frac{d-2}{2}$ such pairs: $$1 +(d-1), 2+(d-2), \ldots, \tfrac{d-2}{2}+\tfrac{d+2}{2}.$$
In both cases, these will lead to ${}_{d-3}F_{d-4}$ hypergeometric terms since three factors will be canceled in the corresponding Gauss sum expression.
\end{proof}

\section{Example: Dwork Threefold}\label{sec:DworkThreefold}

We now prove Theorem \ref{thm:ThreefoldPointCount}. To prove this, we will start with Koblitz's point count formula in \cite{Koblitz}. We then break this down into 6 sets of Gauss sum terms. One of these we have already proved is a ${}_4F_3$ hypergeometric function in \cite{GoodsonDwork2017} . Four of the remaining sets can be rewritten using Proposition \ref{thm:GaussSumIDFn}.  Note that McCarthy gave a $p$-adic hypergeometric point count formula in \cite{McCarthy2012b} for the Dwork threefold that holds for $\lambda=1$. Our formula should match this when we use $\lambda^5=1$.\\ 

This work is very similar to our work with Dwork K3 surfaces in \cite{GoodsonDwork2017}, so we omit some details. We also use some results of McCarthy \cite{McCarthy2012b} that apply here. 

\begin{proof}[Proof of Theorem \ref{thm:ThreefoldPointCount}]
Let $W$ be the set of all 5-tuples $w=(w_1,w_2,w_3,w_4,w_5)$ satisfying $0\leq w_i<5$ and $\sum w_i\equiv 0 \pmod 5$. We denote the points on the diagonal hypersurface
$$x_1^5+\ldots+x_5^5=0$$
by $N_q(0):=\sum N_{q}(0,w)$, where
$$
 N_{q}(0,w)=
 \begin{cases}
  0 &\text{if some but not all } w_i=0,\\
  \frac{q^4-1}{q-1} &\text{if all } w_i=0,\\
  -\frac1q J\left(T^{\tfrac{w_1}{5}},T^{\tfrac{w_2}{5}},\ldots,T^{\tfrac{w_5}{5}}\right) &\text{if all } w_i\not=0.\\
 \end{cases}
$$

Theorem 2 of \cite{KoblitzHypersurface} tells us that the number of points on the Dwork threefold is given by
$$\#X_{\lambda}^5(\mathbb F_q)=N_q(0)+\frac1{q-1}\sum\frac{\prod_{i=1}^5g\left(T^{w_it+j}\right)}{g(T^{5j})}T^{5j}(5\lambda)$$
where the sum is taken over $j\in\{0,\ldots,q-2\}$ and $w=(w_1,w_2,w_3,w_4,w_5)$ in $W/\sim$.\\

We wish to simplify this formula. We start by considering the term $N_q(0)$.

\begin{lem}\label{lemma:N0}
 $N_q(0)=\frac{q^4-1}{q-1}  + 50\sum_{i=1}^4g(T^{it})^2g(T^{3it})+ \frac1q\sum_{i=1}^4g(T^{it})^5$.
\end{lem}
\begin{proof}
McCarthy proves this result in \cite{McCarthy2012b} (see Equation 3.3).
\end{proof}

Define the equivalence relation $\sim$ on $W$ by $w\sim w'$ if $w-w'$ is a multiple of $(1,1,1,1,1)$. Up to permutation, McCarthy shows that there are six cosets in $W^*=W/\sim$:
$$(0,0,0,0,0)^1, (0,1,2,3,4)^{24}, (0,0,0,1,4)^{20}, (0,0,0,2,3)^{20}, (0,0,1,1,3)^{30}, (0,0,1,2,2)^{30}.$$

Let \begin{equation}\label{eqn:DefS_w5}
 S_{[w]}=\frac{1}{q-1}\sum_{j=0}^{q-2}\frac{\prod_{i=1}^5g\left(T^{w_it +j}\right)}{g\left(T^{5j}\right)}T^{5j}(5\lambda)    
    \end{equation}
denote the summands corresponding to all $w'\in[w]$. Our main tool for simplifying terms of this form is the Hasse-Davenport relation for Gauss sums.\\

In the lemmas that follow, we give explicit formulas for each $S_{[w]}$.

\begin{lem}\label{lem:Threefold0000}
Let $w=(0,0,0,0,0)$. Then 
$$S_{[w]}=-\frac{1}{q}\sum_{i=1}^4g\left(T^{it}\right)^5+q^2{}_{4}F_{3}\left(\left.\begin{array}{cccc}
                T^t&T^{2t}&T^{3t}&T^{4t}\\
		{} &\epsilon&\epsilon&\epsilon
               \end{array}\right|\frac{1}{\lambda^5}\right)_q   $$
\end{lem}

\begin{proof}
This is proved by specializing Theorem 8.1 of \cite{GoodsonDwork2017} to the case where $d=5$.

\end{proof}

\begin{lem}\label{lem:Threefold01234}
Let $w=(0,1,2,3,4)$. Then 
$$S_{[w]}=24q^2\delta(1-\lambda^5), $$
where $\delta(x)=1$ if $x=0$ and $\delta(x)=0$ otherwise, and the coefficient of 24 corresponds to the number of permutations of $(0,1,2,3,4)$ that are in distinct cosets.
\end{lem}

\begin{proof}
When $w=(0,1,2,3,4)$, we have
\begin{align*}
    S_{[w]}&=\frac{24}{q-1}\sum_{j=0}^{q-2} \frac{g(T^j)g(T^{t+j})g(T^{2t+j})g(T^{3t+j})g(T^{4t+j})}{g(T^{5j})} T^{5j}(5\lambda).
\end{align*}
We use the Hasse-Davenport relation and properties of Gauss sums to rewrite this and obtain
\begin{align*}
    S_{[w]}&=\frac{24}{q-1}\sum_{j=0}^{q-2} \prod_{i=1}^{4}g(T^{it}) T^{5j}(\lambda)\\
    &=\frac{24q^2}{q-1}\sum_{j=0}^{q-2} T^{5j}(\lambda)\\
    &=24q^2\delta(1-\lambda^5), 
\end{align*}
since
$$
\sum_{j=0}^{q-2} T^{j}(\lambda^5)=
\begin{cases}
q-1&\text{ if }\lambda^5=1,\\
0&\text{ otherwise.}
\end{cases}
$$

\end{proof}

We now work to rewrite the remaining terms. Unlike in our work with Dwork K3 surfaces in \cite{GoodsonDwork2017}, these terms all break down in a similar manner. Thus, we will carefully show our work for $w=(0,0,0,1,4)$ and state the remaining results.

\begin{lem}\label{lem:Threefold00014}
Let $w=(0,0,0,1,4)$. Then 
$$S_{[w]}=-20g(T^{2t})^2g(T^t)-20g(T^{3t})^2g(T^{4t})+ 20q^2{}_{2}F_{1}\left(\left.\begin{array}{cc}
                T^{2t}&T^{3t}\\
		{} &\epsilon               \end{array}\right|\frac{1}{\lambda^5}\right)_q.$$
\end{lem}

\begin{proof}
When $w=(0,0,0,1,4)$, we have
\begin{align*}
    S_{[w]}&=\frac{20}{q-1}\sum_{j=0}^{q-2} \frac{g(T^j)^3g(T^{t+j})g(T^{4t+j})}{g(T^{5j})} T^{5j}(5\lambda).
\end{align*}
We use Hasse-Davenport to write
\begin{align*}
    S_{[w]}&=\frac{20q^2}{q-1}\sum_{j=0}^{q-2} \frac{g(T^j)^2}{g(T^{3t+j})g(T^{2t+j})} T^{5j}(\lambda). 
\end{align*}

Note that when $j=2t$ or $3t$, we have $g(T^0)$ in the denominator. We separate these two cases from the summand and evaluate them to get
\begin{align*}
    \frac{g(T^{2t})^2}{g(T^{3t+2t})g(T^{2t+2t})} T^{10t}(\lambda)&= -\frac1q g(T^{2t})^2g(T^t),\\
    \frac{g(T^{3t})^2}{g(T^{3t+3t})g(T^{2t+3t})} T^{15t}(\lambda)&= -\frac1q g(T^{3t})^2g(T^{4t}).
\end{align*}

For the remaining values of $j$ we use the relationship $g(\chi)g(\overline\chi)=\chi(-1)q$, noting that $T^t(-1)=T^{5t}(-1)=1$.
\begin{align*}
    \sum_{j=0,j\not=2t,3t}^{q-2} \frac{g(T^j)^2}{g(T^{3t+j})g(T^{2t+j})} T^{5j}(\lambda)&=\frac{1}{q^2}\sum_{j=0,j\not=2t,3t}^{q-2} g(T^j)^2g(T^{2t-j})g(T^{3t-j}) T^{5j}(\lambda). 
\end{align*}

Note that for $j=2t,3t$ we have 
\begin{align*}
    g(T^{2t})^2g(T^{2t-2t})g(T^{3t-2t}) T^{10t}(\lambda)&=-g(T^{2t})^2g(T^{t}),\\
    g(T^{3t})^2g(T^{2t-3t})g(T^{3t-3t}) T^{15t}(\lambda)&=-g(T^{3t})^2g(T^{4t}).
\end{align*}

Finally, we use Corollary \ref{thm:GaussSumIDFn} to rewrite the main Gauss sum term.

\begin{align*}
    \sum_{j=0}^{q-2} g(T^j)^2g(T^{2t-j})g(T^{3t-j}) T^{5j}(\lambda)&=q^2(q-1){}_{2}F_{1}\left(\left.\begin{array}{cc}
                T^{2t}&T^{3t}\\
		{} &\epsilon               \end{array}\right|\frac{1}{\lambda^5}\right)_q.
\end{align*}

Hence,
\begin{align*}
    S_{[w]}&=\frac{20q^2}{q-1}\left(-\frac1q g(T^{2t})^2g(T^t)-\frac1q g(T^{3t})^2g(T^{4t})+\frac{1}{q^2}g(T^{2t})^2g(T^{t})+\frac{1}{q^2}g(T^{3t})^2g(T^{4t})\right.\\
    &\hspace{1.5in}+\left. (q-1){}_{2}F_{1}\left(\left.\begin{array}{cc}
                T^{2t}&T^{3t}\\
		{} &\epsilon               \end{array}\right|\frac{1}{\lambda^5}\right)_q\right)\\
	&=\frac{20}{q-1}g(T^{2t})^2g(T^t)(1-q)+ \frac{20}{q-1}g(T^{3t})^2g(T^{4t})(1-q)+ 20q^2{}_{2}F_{1}\left(\left.\begin{array}{cc}
                T^{2t}&T^{3t}\\
		{} &\epsilon               \end{array}\right|\frac{1}{\lambda^5}\right)_q\\
	&=-20g(T^{2t})^2g(T^t)-20g(T^{3t})^2g(T^{4t})+ 20q^2{}_{2}F_{1}\left(\left.\begin{array}{cc}
                T^{2t}&T^{3t}\\
		{} &\epsilon               \end{array}\right|\frac{1}{\lambda^5}\right)_q	
\end{align*}

\end{proof}

Similarly, we have

\begin{align*}
S_{[0,0,0,2,3]}&=-20g(T^{t})^2g(T^{3t})-20g(T^{4t})^2g(T^{2t})+ 20q^2{}_{2}F_{1}\left(\left.\begin{array}{cc}
                T^{t}&T^{4t}\\
		{} &\epsilon               \end{array}\right|\frac{1}{\lambda^5}\right)_q,\\
S_{[0,0,1,1,3]}&=-30g(T^{3t})^2g(T^{4t})-30g(T^{2t})^2g(T^{t})+ 30q^2{}_{2}F_{1}\left(\left.\begin{array}{cc}
                T^{t}&T^{3t}\\
		{} &T^{4t}               \end{array}\right|\frac{1}{\lambda^5}\right)_q,\\
S_{[0,0,1,2,2]}&=-30g(T^{t})^2g(T^{3t})-30g(T^{4t})^2g(T^{2t})+ 30q^2{}_{2}F_{1}\left(\left.\begin{array}{cc}
                T^{t}&T^{2t}\\
		{} &T^{3t}               \end{array}\right|\frac{1}{\lambda^5}\right)_q.
\end{align*}

We now combine all of these terms to get the complete point count formula for the Dwork threefold. Note that the extra Gauss sum terms from $N_q(0)$ will cancel with the extra Gauss sum terms from the $S_{[w]}$ terms. Hence,

\begin{align*}
    \#X_{\lambda}^5(\mathbb F_q)&=\frac{q^{4}-1}{q-1}+ 24 q^2\delta(1-\lambda^5)+ q^{3}{}_{4}F_{3}\left(\left.\begin{array}{cccc}
                T^t&T^{2t}&\ldots &T^{4t}\\
		{} &\epsilon&\ldots&\epsilon \end{array}\right|\frac{1}{\lambda^5}\right)_q\\\\
              \hspace{1in}&+20q^2{}_{2}F_{1}\left(\left.\begin{array}{cc}
                T^{2t}&T^{3t}\\
		{} &\epsilon               \end{array}\right|\frac{1}{\lambda^5}\right)_q+20q^2{}_{2}F_{1}\left(\left.\begin{array}{cc}
                T^{t}&T^{4t}\\
		{} &\epsilon               \end{array}\right|\frac{1}{\lambda^5}\right)_q\\\\
		\hspace{1in}&+30q^2{}_{2}F_{1}\left(\left.\begin{array}{cc}
                T^{t}&T^{3t}\\
		{} &T^{4t}               \end{array}\right|\frac{1}{\lambda^5}\right)_q+30q^2{}_{2}F_{1}\left(\left.\begin{array}{cc}
                T^{t}&T^{2t}\\
		{} &T^{3t}               \end{array}\right|\frac{1}{\lambda^5}\right)_q.
\end{align*}

\end{proof}

\section*{Acknowledgements}
I would like to thank Benjamin Brubaker for helpful conversations while working on these results. I also thank the reviewer for their very thorough and helpful comments. Lastly, I would like to thank Doris McCarthy for her many years of unwavering support, encouragement, and interest in my work. Thank you for thinking of me.

\bibliographystyle{plain}
\bibliography{DworkHyp}

\end{document}